\documentclass[12pt]{amsart}
\usepackage{amsmath}
\usepackage{amsfonts}
\usepackage{amssymb}
\usepackage{graphicx} 
\usepackage{MnSymbol}

\numberwithin{equation}{section}
 
 \newtheorem{thm}{Theorem}[section]
 \newtheorem{lem}[thm]{Lemma}
 \newtheorem{exam}[thm]{Example}
 \newtheorem{prop}[thm]{Proposition}
 \newtheorem{cor}[thm]{Corollary}
 \newtheorem{rem}[thm]{Remark}
 \newtheorem{defn}[thm]{Definition}

 \def\Id{\mathop{\rm Id}\nolimits}
 \def\re{\mathop{\rm Re\,}\nolimits}
 \def\dist{\mathop{\rm dist}\nolimits}
 \def\Im{\mathop{\rm Im}\nolimits}

\newcommand{\A}{{\mathcal A}}   
\newcommand{\C}{{\mathbb C}}   
 \newcommand{\R}{{\mathbb R}}           
\newcommand{\N}{{\mathbb N}}        
\newcommand{\K}{{\mathbb K}}        
\newcommand{\LX}{{\mathcal L}(X)}   
\newcommand{\LH}{{\mathcal L}(H)}  
\newcommand{\LY}{{\mathcal L}(Y)}  
 \newcommand{\KX}{{\mathcal K}(X)}    

\author{W. Arendt}
\address{Wolfgang Arendt, Institute of Applied Analysis, University of Ulm. Helmholtzstr. 18, D-89069 Ulm (Germany)} 
\email{wolfgang.arendt@uni-ulm.de}

\author{I. Chalendar}
\address{Isabelle Chalendar,  Universit\'e Gustave Eiffel, LAMA, (UMR 8050), UPEM, UPEC, CNRS, F-77454, Marne-la-Vallée (France)}
\email{isabelle.chalendar@u-pem.fr}

\title[Essentially Coercive Forms]{Essentially Coercive Forms and asympotically compact semigroups}
\begin{document}	
\begin{abstract}   
Form methods are most efficient to prove generation theorems for semigroups but also for proving selfadjointness. So far those theorems are based on a coercivity notion which allows the use of the Lax-Milgram Lemma. Here we consider weaker "essential" versions of coerciveness which already suffice to obtain the generator of a semigroup $S$ or a selfadjoint operator. We also show that one of these properties, namely \emph{essentially positive coerciveness} implies a very special asymptotic behaviour of $S$, namely \emph{asymptotic compactness}; i.e. that $\dist(S(t),{\mathcal K}(H))\to 0$ as $t\to\infty$, where ${\mathcal K}(H)$ denotes the space of all compact operators on the underlying Hilbert space. 
\end{abstract}	
\maketitle
    
\section{Introduction}\label{sec:1}
Form methods are most efficient for proving well-posedness results for parabolic equations. We refer to the monography of Kato \cite{Ka80} and Ouhabaz \cite{Ou} for example. 

The scenario is the following. Given are Hilbert spaces $H$, $V$ such that $V\underset{d}{\hookrightarrow} H$ (i.e. $V$ is densely and continuously embedded into $H$) and a continuous sesquilinear form $a:V\times V\to\C$. Then there is a unique operator $A$ on $H$ whose graph is given by  
\[     G(A)  =\{ (u,f):u\in V,a(u,v)=\langle f,v\rangle_H\mbox{ for all }v\in V\}.\] 
We call $A$ the \emph{operator associated with $a$} and write $A\sim a$. The following is a classical result. Assume that $a$ is \emph{positive-coercive}, i.e. 
\[   \re a(u,u) \geq \alpha \|u\|_V^2\mbox{ for all }u\in V\mbox{ and some }\alpha>0.  \] 
Then $-A$ generates a holomorphic $C_0$-semigroup $S$ which is \emph{uniformly exponentially stable}, i.e. 
\[     \|S(t)\|\leq Me^{-\delta t} \mbox{ for all }t>0\mbox{ and for some }\delta>0,M\geq 0. \]
The purpose of this paper is to relax the condition of positive coerciveness as well as similar notions and to replace them  by a weaker topological notion, which we call "essential versions". 
\begin{defn}\label{def:1.1}
A continuous and sesquilinear form $a$ on $V\times V$ is called \emph{essentially positive-coercive} if 
\[ u_n\rightharpoonup 0\mbox{ in }V\mbox{ and }\limsup_{n\to\infty}\re a(u_n,u_n)\leq 0  \]
implies $\|u_n\|_V\to 0$ as $n\to\infty$. 
\end{defn}
Here $u_n\rightharpoonup 0$ means that $u_n$ converges weakly to $0$ in $V$.
Then we show the following.
\begin{thm}\label{th:1.2}
Assume that $a$ is essentially positive-coercive. Then $-A$ generates a holomorphic $C_0$-semigroup $S$ which is asymptotically compact. 
\end{thm}
Here we call $S$ \emph{asymptotially compact} (sometimes called \emph{quasicompact} in the literature) if 
	\[   \|S(t)\|_{ess}\leq Me^{-\delta t} \mbox{ for all }t\geq 0,  \]
where, for $T\in\LH$, $\|T\|_{ess}$ is the distance of $T$ to the space of all compact operators on the Hilbert space $H$, denoted by ${\mathcal K}(H)$. We also
 show that each asymptotically compact quasi-contractive holomorphic semigroup is obtained via a form in this way if we allow to pass to an equivalent scalar product on $H$. Theorem~\ref{th:1.2} seems of interest for two reasons. A specific topological property of the form $a$ is responsible for the fact that the semigroup behaves like a finite dimensional system at infinity. But also the pure generation property is surprising since no range condition on the operator is needed. Much more generally, we investigate the \emph{numerical range}
 \[     W(a):=\{     a(u,u):u\in V,\|u\|_H=1\}\]
of the form  $a$ (which is a convex set). 

We call the form $a$ \emph{coercive} if 
\[  |a(u,u)|\geq \alpha \|u\|_V^2  \mbox{ for all }u\in V\mbox{ and some }\alpha>0.    \]    
Note that the real part in the definition of positive-coerciveness is replaced by the absolute value. 

We call $a$ \emph{essentially coercive} if 
\[ u_n\rightharpoonup 0\mbox{ in }V\mbox{ and }a(u_n,u_n)\to 0\mbox{ implies }\|u_n\|_V\to 0.      \] 
Then we show in Section~\ref{sec:2} that 
\begin{equation*}\label{eq:1.1}
\sigma(A)\subset \overline{W(a)}=\overline{W(A)},
\end{equation*}	
whenever $a$ is essentially coercive.

This spectral inclusion is remarkable  since it fails in general for unbounded operators. It implies generation results. For example we show the following: if $a$ is essentially coercive and accretive, then $-A$ is m-accretive. In other words, again, the  range condition comes automatically; i.e. it is a consequence of essential coerciveness.   We investigate in particular the case where $a$ is symmetric. Letting 
\[   a_\lambda(u,v):=a(u,v)-\lambda \langle u,v \rangle_H,  \]	
we prove the following. 
\begin{thm}\label{th:1.3}
Assume that $a$ is symmetric and that there exists $\lambda\in\C$ such that $a_\lambda$ is essentially coercive. Then $A$ is selfadjoint and semibounded. Moreover, if $\lambda\in\R$, then \[\sigma_{ess}(A)\subset (-\infty,\lambda)\mbox{ or }\sigma_{ess}(A)\subset (\lambda,\infty).\]  
\end{thm}
We also show that each semibounded selfadjoint operator can be obtained in this way by a unique essentially coercive form. 

We should say some words about preceding results. The notion of essential coerciveness  has been 
introduced in \cite{ace}, where it is shown that this property is equivalent to the convergence of arbitrary Galerkin approximations defined by the form. So the motivation and applications (to finite elements) were completely different in this preceding paper.  A predecessor with investigations on semigroups is the paper \cite{AtEKS14} where compactly elliptic forms play a role for investigating the Dirichlet-to-Neumann operator. One of our result shows that compact ellipticity is actually equivalent to essential positive coerciveness.
We also mention that  ter Elst, Sauter and Vogt \cite{tESV15} established form methods which lead to contractive semigroups which are not necessarily holomorphic. However, the use of topological conditions such as essential coerciveness (Definion~\ref{def:1.1}) seems to be new.

The paper is organized as follows:
\tableofcontents

\section{Numerical range and essential coercivity}\label{sec:2}

Let $V$ be  a Hilbert space over $\K=\R$ or $\C$.


The notion of essential coerciveness has been introduced in \cite{ace} to study Galerkin approximation. For our purposes, the following Fredholm-property plays a crucial role. 
\begin{prop}\label{propace:2.1}[\cite[Corollary 4.5]{ace}]
Let $a:V\times V\to \K$ be a continuous, essentially coercive form which \emph{satisfies uniqueness} i.e. 
\begin{equation}\label{eq:uni}
\mbox{for }u\in V, \, a(u,v)=0\mbox{ for all }v\in V\mbox{ implies }u=0.
\end{equation}   
Then for all $f\in V'$, there exists a unique $u\in V$ such that $a(u,v)=\langle f,v\rangle$ for all $v\in V$.  
\end{prop}	   
Now let $H$ be a second Hilbert space over $\K$ and $j\in {\mathcal L}(V,H)$ with dense range. Let $a:V\times V\to \K$ be a continuous 
sesquilinear form satisfying 
\begin{equation}\label{eq:propj}
\mbox{ if }u\in\ker j \mbox{ and }a(u,v)=0 \mbox{ for all }v\in\ker j\mbox{ then }u=0.
\end{equation}
Then there exists a unique operator $A$ on $H$ whose graph is given by 
\[ G(A)= \{  (x,f)\in H\times H:\exists u\in V,j(u)=x\mbox{ and } a(u,v)=\langle f,j(v)\rangle_H, v\in V   \}  \]
We call $A$ \emph{the operator on $H$ associated with $(a,j)$} and write $A\sim (a,j)$. 

This setting has been introduced in \cite{AtE12}, generalizing the common case where  $V\underset{d}{\hookrightarrow} H$ (i.e. $V$ embeds continuously in $H$ with dense range) for which $j$ is the identity mapping. See  also \cite{AtE} for an introduction and \cite{Sau,AtEKS14} for more information.  In general $A$ is not a closed operator. Nevertheless, if $a$ is essentially coercive, then we can prove quite strong spectral properties.

By 
\begin{equation*}\label{eq:im}
W(a,j):=\{   a(u,u):u\in V; \|j(u)\|_H=1\}
\end{equation*}  
we denote \emph{the numerical range} of $(a,j)$. For $\lambda\in\K$, we let 
\begin{equation*}\label{eq:pert}
a_{\lambda} (u,v)=a(u,v)-\lambda \langle j(u),j(v)\rangle_H . 
\end{equation*} 
Thus $a_\lambda:V\times V\to \K$ is a continuous sesquilinear form.
 
We first prove a lemma.
\begin{lem}\label{lem:perturb}
Let $\lambda\in\K\backslash \overline{W(a,j)}$, where $\overline{W(a,j)}$ denotes the closure of $W(a,j)$. If $a$ is essentially coercive then so is $a_\lambda$. 
\end{lem}
\begin{proof}
Let $u_n\rightharpoonup 0$ in $V$  and assume that $a_\lambda (u_n,u_n)\to 0$ as $n\to\infty$. Then $\|j(u_n)\|_H\to 0$. Indeed, if not, there exist a subsequence $(u_{n_k} )_k$ and $\delta>0$ such that $\|j(u_{n_k})\|_H\geq \delta$ ($k\in\N$). Let $w_k=u_{n_k}/\|j(u_{n_k})\|_H$.  Then 
\[a_\lambda (w_k,w_k)=\frac{1}{\|j(u_{n_k})\|_H^2}a_\lambda (u_{n_k},u_{n_k})\to 0\mbox{ as }n\to\infty.\]
But $a_\lambda(w_k,w_k)=a(w_k,w_k)-\lambda$ and $a(w_k,w_k)\in W(a,j)$. This contradicts the fact that $\lambda\not\in \overline{W(a,j)}$. 

It follows that $a(u_n,u_n)=a_\lambda (u_n,u_n) +\lambda \|j(u_n)\|_H^2\to 0$ as $n\to\infty$.  Now since $a$ is essentially coercive, we get $\lim_{n\to\infty}\|u_n\|_V=0$. 
\end{proof}	
We denote by $\rho(A)$ the resolvent set of $A$ and for $\mu\in \rho(A)$ by $R(\mu,A)=(\mu\Id -A)^{-1}$ the resolvent at $\mu$. Moreover $\sigma (A):=\K\backslash \rho (A)$ is the spectrum of $A$. The following is the main result of this section.  
\begin{thm}\label{th:new23}
	Assume that there exists $\lambda\in\K$ such that $a_\lambda$ is essentially coercive. Then 
	\[\sigma(A)\subset \overline{W(a,j)}\mbox{ and }\|R(\mu, A)\|\leq \frac{1}{\dist (\mu, W(a,j))}\mbox{ for all }\mu\not\in \overline{W(a,j)}.\] 
\end{thm}
\begin{proof}
a) Assume that $\lambda=0$. Let $\mu\in \K\backslash \overline{W(a,j)}$. Then $a_\mu$ is essentially coercive by Lemma~\ref{lem:perturb}. We show that $a_\mu$ satisfies the uniqueness condition (\ref{eq:uni}). Let $u\in V$ such that 
\[a_\mu(u,v)=a(u,v)-\mu \langle j(u),j(v)\rangle_H=0\mbox{ for all }v\in V.\]
If $j(u)\neq 0$, then $a(w,w)-\mu=0$ where $w=\frac{u}{\|j(u)\|_H}$. Thus $\mu\in W(a,j)$, in contradiction with the assumption. Thus $j(u)=0$ and it follows from (\ref{eq:propj}) that $u=0$.   

Let $f\in H$. By Proposition~\ref{propace:2.1} there exists a unique $u\in V$ such that 
\[ a_\mu(u,v)=\langle -f,j(v)\rangle_H \mbox{ for all }v\in V.   \]
It follows that 
\[  a(u,v)=\langle -f +\mu j(u),j(v)  \rangle_H  \mbox{ for all }v\in V. \]
Thus $x:=j(u)\in D(A)$ and $Ax=-f+\mu x$. Moreover, if $f\neq 0$, \[a(w,w)-\mu=\langle -f,\frac{j(w)}{\|j (u)\|_H}\rangle_H,\mbox{ with }  w:=\frac{u}{\|j(u)\|_H}. \] 
It follows that 
\[ \dist (\mu, W(a,j))\leq |\langle f,j(w)\rangle_H|\frac{1}{\|j(u)\|_H}\leq \frac{ \|f\|_H }{\|j(u)\|_H}.      \]
Hence \[\|x\|_H=\|j(u)\|_H\leq \frac{1}{\dist (\mu, W(a,j))}\|f\|_H.\]
\noindent b) Let $\lambda\in\K\backslash \{0\}$ such that $a_\lambda$ is essentially coercive. Let $\mu\in \K\backslash \overline{W(a,j)}$. Observe that 
\[  \overline{W(a_\lambda ,j)} =\overline{W(a ,j)}-\lambda.   \]
Thus $\mu-\lambda\not\in \overline{W(a_\lambda ,j)}$. Note also that $(a_\lambda,j)\sim A-\lambda\Id$.    	It follows from a) that 
$ (\mu-\lambda)\Id -(A-\lambda\Id) =\mu\Id-A $ is invertible  and 
\begin{eqnarray*}
	 \| (\mu\Id-A)^{-1}\|=\| ((\mu-\lambda)\Id-(A-\lambda \Id))^{-1}\| & \leq &  \frac{1}{\dist (\mu-\lambda, W(a_\lambda,j))}\\
	  & = & \frac{1}{\dist(\mu,W(a,j))}. 
\end{eqnarray*}	  
\end{proof}
We mention that Theorem~\ref{th:new23} fails in general if the form is not essentially coercive. Here is an example which is well-known (see \cite[Ex. 5.3]{tubingen} for more information). 
\begin{exam}
	Let $H=L^2(0,\infty)$, $V=H^1_0(0,\infty)$, $j$ the identity and $a(u,v)=\int_0^\infty u'\overline{v}$.  Then $W(a,j)\subset i\R$ as can easily be seen by integration by parts. Here the associated  operator $A$ on $H$ is given by $D(A)=V$, $Af=-f'$. Thus $-A$ is the generator of the right shift semigroup and its spectrum is $\sigma(A)=\{ \lambda\in \C:\re (\lambda)\geq 0\}$.     
\end{exam}  
\begin{rem}
	Given a closed operator on $H$ we always might choose $V=D(A)$ with the graph norm $\|u\|^2_A=\|u\|_H^2 +\|Au\|_H^2$ and define $a(u,v)=\langle Au,v\rangle$. If $a$ is essentially coercive, then \[\sigma(A)\subset \overline{W(a)}:=\overline{\{  \langle Au,u\rangle_H:u\in D(A)\}}.\]
	But this $a$ is rarely essentially coercive. In fact, if the injection $V\hookrightarrow H$ is compact, then $a$ is essentially coercive if and only if $\dim V<\infty$.    
\begin{proof}
Let $u_n\ \rightharpoonup 0$  in $V$. Then $Au_n\rightharpoonup 0$ in $H$ by the definition of the norm in $V$. Moreover, $u_n\to 0$ in $H$ since the embbeding is compact. It follows that $\langle Au_n,u_n\rangle \to 0$. Since $a$ is essentially coercive we conclude that $\|u_n\|_V\to 0$. In other words, we have shown that each weakly convergent sequence in $V$ is norm convergent and thus $\dim V<\infty$.  
\end{proof}	
But in this situation, it may well happen that 
\[\sigma(A)\subset \overline{\{  \langle Au,u\rangle_H : u\in D(A),\|u\|_H=1  \}}\]
 (for example if $A$ is selfadjoint with compact resolvent). Thus the essential coercivity is not a necessary condition in Theorem~\ref{th:new23}. 
\end{rem}
Next we want to use Theorem~\ref{th:new23} to prove several generation theorems. The first concerns selfadjoint operators, the second contraction semigroups and the third holomorphic semigroups. The point is that in the usual versions of the Lumer-Philipps Theorem, Theorem~\ref{th:new23} allows us to replace the range condition by essential coercivity. 

 \section{Selfadjoint, m-accretive and m-quasi-sectorial operators}\label{sec:3}
\subsection{Selfadjoint  operators}
Let $H$ be a Hilbert space over $\C$. An operator $A$ on $H$ is called \emph{symmetric} if 
\[  \langle Ax,y\rangle =\langle x,Ay\rangle \mbox{ for all }x,y\in D(A).  \]
An operator $A$ is \emph{selfadjoint} if it is densely defined and $A=A^*$. By a well-known criterion, 
\begin{center}
$A$ is selfadjoint if and only if it is symmetric and $\pm i\Id-A:D(A)\to H$ are surjective. 
\end{center}  
Now we consider a Hilbert space $V$ and a mapping $j\in {\mathcal L}(V,H)$ with dense range. Let $a:V\times V\to \C$ be a continuous sesquilinear form satisfying  (\ref{eq:propj}). Let $A\sim (a,j)$ and $a_\lambda (u,v):=a(u,v)-\lambda \langle j(u),j(v)\rangle_H$ ($\lambda\in\C$) as in Section~\ref{sec:2}.
\begin{thm}\label{th:new26}
	 Assume that $a$ is symmetric. If $a_\lambda$ is essentially coercive for some $\lambda\in\C$, then $A$ is selfadjoint.   
\end{thm}
\begin{proof}
It follows from the definition that $A$ is symmetric. By Theorem~\ref{th:new23}, $\sigma(A)\subset \overline{W(a,j)}\subset\R$. This implies that $A$ is selfadjoint.  	
\end{proof}	
We will show in Section~\ref{sec:8} that the operator $A$ in Theorem~\ref{th:new26} is semibounded if we assume in addition that $j$ is injective. In that case, each semibounded selfadjoint operators is obtained in that way.

We give an example to illustrate how Theorem~\ref{th:new26} can be used. It is convenient (even though surprising) that only one complex number, for instance $\lambda=i$, suffices to check the essential coercivity. 
\begin{exam}
	Let $H=L^2(\R^d)$, $d\geq 3$, $m\in L^d(\R^d)$ real-valued, $V=H^1(\R^d)$ and $j:V\to H$ the identity. Then 
	\[   a(u,v)=\int \nabla u\overline{\nabla v} +\int mu\overline{v}   \]
	defines a continuous, symmetric sesquilinear form on $V\times V$. In fact, by Sobolev-embedding, $H^1(\R^d)\subset L^{2d/(d-2)}(\R^d)$. Since $m\in L^d(\R^d)$, it follows that 
	\begin{equation}\label{eq:new25}
	mL^{2d/(d-2)}(\R^d)\subset L^2(\R^d).
	\end{equation}    
We show that 
\[   a_i(u,v):=a(u,v)-i\langle u,v\rangle_H \]
is essentially coercive. In fact, consider a sequence $(u_n)_n$ such that  $u_n\rightharpoonup 0$ in $H^1(\R^d)$ and $a_i(u_n,u_n)\to 0$. Then $\|u_n\|^2_{L^2}=-\Im a_i(u_n)\to 0$ as $n\to \infty$. Since $u_n\rightharpoonup 0$ in $H^1(\R^d)$, by (\ref{eq:new25}), $mu_n\rightharpoonup 0$ in $L^2(\R^d )$. Moreover, since $u_n\to 0$ in $L^2(\R^d)$, this implies that 
\[\langle mu_n,u_n\rangle_H=\int m|u_n|^2 \to 0\mbox{ as }n\to\infty.\]
Consequently
\[  \int|\nabla u_n|^2 = \re a_i(u_n)-\int m|u_n|^2\to 0\mbox{ as }n\to\infty.   \]
We have shown that $\|u_n\|_{H^1}\to 0$ and thus $a_i$ is essentially coercive. It follows from Theorem~\ref{th:new26} that the operator $A$ associated with $(a,j)$ on $H$ is selfadjoint. Using    (\ref{eq:new25}) one sees that 
\[D(A)=H^2(\R^d) \mbox{ and }Au=\Delta u+mu\mbox{ for all }u\in D(A).\] 	
\end{exam}
We next present an example where $j$ is not injective.  In fact, we consider the Dirichlet-to-Neumann operator associated to the problem 
\[   -\Delta u +mu\]
where $m$ is a measurable function. This was the prototype example in \cite{AtEKS14}, where $m$ had been chosen bounded. Here we allow more general $m$ to which Theorem~\ref{th:new26} can be applied conveniently. 
\begin{exam}\label{ex:3.4}{(Dirichlet to Neumann operator)}
Let $\Omega\subset\R^d$ be a bounded, open Lipschitz domain, $d\geq 3$, $V=H^1(\Omega)$, $m\in L^d(\Omega)$. Then 
\[   a(u,v)=\int_{\Omega}  \nabla u\overline{\nabla v} +\int_{\Omega} mu\overline{v}    \]
defines a continuous, symmetric form on $V$. We claim that $a$ is essentially coercive.
\begin{proof}
Let $u_n\hookrightarrow 0$ in $H^1(\Omega)$ such that $a(u_n,u_n)\to 0$. Since $H^1(\Omega)\hookrightarrow L^{2d/(d-2)}$, there exists $c\geq 0$ such that $\|u_n\|_{L^{2d/(d-2)}}\leq c$ for all $n\in\N$. 
	
Since the embedding $H^1(\Omega)\hookrightarrow L^2(\Omega) $ is compact, $u_n\to 0$ in $L^2(\Omega)$ as $n\to\infty$.  Thus 
\begin{eqnarray*}
\left|    \int_{\Omega} m|u_n|^2 \right| & \leq & \|u_n\|_{L^2(\Omega)}\left( \int_{\Omega}|m|^2|u_n|^2\right)^{1/2}\\
 & \leq & \|u_n\|_{L^2(\Omega)} \left( \int_{\Omega}|m|^{2\frac{d}{2}} \right)^{\frac{2}{d}\frac{1}{2}}  	 \left( \int_{\Omega}|m|^{2\frac{d}{d-2}} \right)^{\frac{d-2}{d}\frac{1}{2}}  \\
  & \leq &  \|u_n\|_{L^2(\Omega)}  \|m\|_{L^d(\Omega)} c\to 0
\end{eqnarray*}	  
as $n\to\infty$. 
Hence $\int_{\Omega}|\nabla u_n|^2\to 0$ as $n\to\infty$. 
We have shown that $\|u_n\|_{H^1(\Omega)}\to 0$ as $n\to\infty$.  
\end{proof} 
Now let $H=L^2(\partial\Omega)$ and let
\[j=\mathop{tr}:H^1(\Omega)\to L^2(\partial \Omega)\]
be the trace operator. Denote by $A$ the operator associated with $(a,j)$ on $L^2(\partial \Omega)$.Then $A$ is selfadjoint by Theorem~\ref{th:new26}. 

We claim that 
\begin{eqnarray*}
	G(A) & = & \{  (g,h)\in L^2(\partial\Omega)\times L^2(\partial \Omega):\exists u\in H^1(\Omega),\\
	 &  &  -\Delta u+mu=0,\mathop{tr} u=g, \partial_\nu u=g \}.
\end{eqnarray*}	 
\begin{proof}
$\subset$ Let $(g,h)\in G(A)$. By definition there exists $u\in H^1(\Omega)$ such that $\mathop{tr} u=g$ and \[a(u,v)=\int_{\partial\Omega} h\overline{\mathop{tr}(v)}\]
 for all $v\in H^1(\Omega)$. Choosing $v\in {\mathcal D}(\Omega)$, one sees that $-\Delta u+m u=0$. Thus 
 \[  \int_{\partial\Omega} \Delta u \overline{v} + \int_{\partial\Omega}\nabla u\overline{\nabla v}=a(u,v)=\int_{\partial\Omega} h\overline{\mathop{tr}(v)}  \]
 for all $v\in H^1(\Omega)$. This means by definition that $\partial_\nu u=h$ in the weak sense.\\
 $\supset$ The proof is similar.     	
\end{proof}	
\end{exam}

The characterization of  selfadjoint  operators which are associated with an essentially coercive form will be given in Section~\ref{sec:6} (see Corollary~\ref{cor:6.6} and \ref{cor:6.5}).

\subsection{m-accretive operators}\label{sec:new4}
Let $\K=\R$ or $\C$ and let $H$ be a Hilbert space. An operator $A$ on $H$ is called \emph{accretive} if $\re \langle Au,u\rangle_H\geq 0$ for all $u\in D(A)$. The operator $A$ is called \emph{m-accretive} if it is accretive and $\Id +A:D(A)\to H$ is surjective. By the Lumer-Philipps Theorem $A$ is m-accretive  if and only if $-A$ generates  contractive $C_0$-semigroup on $A$. Let $V$ be a Hilbert space over $\K$ and $j\in {\mathcal L}(V,H)$ with dense range. Let $a:V\times V\to \K$ be a continuous sesquilinear form satisfying  (\ref{eq:propj}). We say that $a$ is \emph{accretive} if 
\[\re a(u,u) \geq 0 \mbox{ for all }u\in V. \] 
Denote by $A$ the operator on $H$ associated with $(a,j)$. 
\begin{thm}\label{th:new29}
	Assume that $a$ is accretive and $a_\lambda$ is essentially coercive for some $\lambda\in\K$. Then $A$ is m-accretive. 
\end{thm} 
\begin{proof}
 It follows from the definition of the operator associated with $(a,j)$ that $A$ is accretive. By Theorem~\ref{th:new23}, 
 \[\sigma(A)\subset \overline{W(a,j)}	\subset \{\lambda\in\C:\re(\lambda)\geq 0\}.\]
 Thus $-1\not\in \sigma(A)$. 
\end{proof} 
\subsection{m-sectorial operators}
Let $V,H$ be complex Hilbert spaces and let $j\in{\mathcal L}(V,H)$ have dense range. Let $a:V\times V\to\C$ be a continuous sesquilinear form satisfying  (\ref{eq:propj}). We say that $a$ is \emph{$j$-sectorial} if there exist $\theta\in\R$ and $w\in\R$ such that 
\begin{equation}\label{eq:new23}
a(u,u)+w\|j(u)\|_H^2\in\Sigma_\theta\mbox{ for all }u\in V,
\end{equation}    
where $\Sigma_\theta :=\{ re^{i\alpha} :r>0,|\alpha|<\theta  \} $. 	
\begin{thm}\label{th:new210}
Assume that 
\begin{itemize}
	\item[a)] $a$ is $j$-sectorial;
	\item[b)] $a_\lambda$ is essentially coercive for some $\lambda \in\C$.
\end{itemize}
 Then the operator $A$ associated with $(a,j)$ on $H$ is m-sectorial.   
\end{thm}
\begin{proof}
Replacing $a(u,v)$ by $a(u,v)+w\langle j(u),j(v)\rangle_H$ for $u,v\in V$, we may assume that $w=0$ in (\ref{eq:new23}). Let $\theta'=\frac{\pi}{2}-\theta$. Then $\re e^{\pm i\theta'}a(u,u)\geq 0$ for all $u\in V$. It follows from Theorem~\ref{th:new29} that $-e^{\pm i \theta'}A$ is the  generator of a contractive $C_0$-semigroup. This implies that $-A$ generates a holomorphic $C_0$-semigroup which is contractive on $\Sigma_{\theta'}$. 
\end{proof}
\section{Structure theorems for essential coerciveness }\label{sec:5} 
Throughout this section $V$ is a  complex  Hilbert space of infinite dimension and $a:V\times V\to \C$ is a continuous sesquilinear form. \\

Recall that $a$ is called \emph{coercive} if there exists $\alpha>0$ such that 
\begin{equation*}\label{eq:51}
| (a(u,u))|\geq \alpha \|u\|_V^2\mbox{  for all }u\in V. 
\end{equation*}   
We call  $a$  \emph{real-coercive} if there exists $\alpha>0$ such that 
\begin{equation*}\label{eq:52}
|\re (a(u,u))|\geq \alpha \|u\|_V^2\mbox{  for all }u\in V 
\end{equation*}    
and \emph{positive-coercive} if there exists $\alpha>0$ such that 
\begin{equation*}\label{eq:53}
\re (a(u,u))\geq \alpha \|u\|_V^2\mbox{  for all }u\in V. 
\end{equation*}  
Recall that the numerical range of $a$ (with respect to  $V$) 
\[   W(a,V):=\{   a(u,u): \|u\|_V=1\}\]
is convex, see \cite{gusta}. This implies the following. 
\begin{prop}\label{prop:51}
	a) The form $a$ is coercive if and only if there exists $\theta\in\R$ such that $e^{i\theta} a $ is positive-coercive.\\
	b) The form $a$ is real-coercive if and only if $a$ or $-a$ is positive-coercive.   
\end{prop}
Now, as in the coercive case (see Section~\ref{sec:2}) we want to introduce topological properties which are weaker than real-coercive or positive-coercive. 
\begin{defn}
	a)  The form $a$ is called  \emph{essentially real-coercive} if 
	\[u_n\rightharpoonup 0\mbox{  in }V \mbox{ and }\re a(u_n,u_n)\to 0\mbox{ implies }\|u_n\|_V\to 0.\] 
	b) The form $a$ is  called  \emph{essentially positive-coercive} if 
	\[u_n\rightharpoonup 0\mbox{  in }V \mbox{ and }\limsup_{n\to\infty} a(u_n,u_n)\leq 0\mbox{ implies }\|u_n\|_V\to 0.\]  	
\end{defn}  
It is obvious that 
\[\mbox{ess. positive-coercive }\Rightarrow\mbox{ ess. real-coercive } \Rightarrow\mbox{ ess. coercive, }   \]
where ess. is the abbreviation of essentially. \\

For the following  proof of the essential analogue of 
Proposition~\ref{prop:51} a) we use a result from \cite{ace}. 
\begin{prop}\label{prop:53}
	The form $a$ is essentially coercive if and only if there exists $\theta\in\R$ such that $e^{i\theta}a$ is essentially positive-coercive.  
\end{prop}
\begin{proof}
	Assume that $a$ is essentially coercive. By \cite[Theorem 4.4]{ace}, there exists a compact operator $K:V\to V'$ such that the form $b$ given by 
	\[  b(u,v)=a(u,v) +\langle Ku,v\rangle \]
	is coercive. Thus, by Proposition~\ref{prop:51}, there exist $\theta\in\R$ and $\alpha>0$ such that 
	\[  \re e^{i\theta}(a(u,u)+\langle Ku,u\rangle )\geq \alpha \|u\|_V^2 \mbox{ for all }u\in V.  \]  	
	This implies that $e^{i\theta}a$ is essentially positive-coercive. In fact, consider  $u_n\rightharpoonup 0$ in $V$  such that $\limsup_{n\to\infty} \re (e^{i\theta}a(u_n,u_n))\leq 0$. Since $K$ is compact, $\|Ku_n\|_{V'}\to 0$. Consequently, $\langle Ku_n,u_n\rangle_V\to 0$ as $n\to\infty$. We conclude by
	\[ \alpha \limsup_{n\to\infty} \|u_n\|_V^2\leq \limsup_{n\to\infty}\re (e^{i\theta} a(u_n,u_n))\leq 0.    \]
\end{proof}
We will see that also Proposition~\ref{prop:51} b) has an essential analogue. 

Before that we give the following characterization of essential positive coerciveness which shows that this property is indeed the same as positive coerciveness after compact perturbation. Recall that $\dim V=\infty$.      
\begin{thm}\label{th:54}
	The following assertions are equivalent:
	\begin{itemize}
		\item[(i)] the form $a$ is essentially positive-coercive; 
		\item[(ii)] there exist a finite dimensional subspace $V_1$ of $V$ and $\alpha>0$ such that 
		\[    \re a(u,u)  \geq \alpha \|u\|_V^2 \mbox{ for all }u\in V_1^\perp;\]
		\item[(iii)] there exist a  finite rank operator $K:V\to V'$ and $\alpha>0$ such that 
		\[
		\re (a(u,u) +\langle Ku,u\rangle)\geq \alpha \|u\|_V^2\mbox{ for all }u\in V;
		\]
		\item[(iv)]  there exist a  compact operator $K:V\to V'$ and $\alpha>0$ such that 
		\[
		\re (a(u,u) +\langle Ku,u\rangle)\geq \alpha \|u\|_V^2\mbox{ for all }u\in V.
		\]
		\item[(v)]  there exist a Hilbert space $Y$, a compact operator $K:V\to Y$ and $\alpha>0$ such that 
		\[
		\re (a(u,u) +\|Ku\|_Y^2)\geq \alpha \|u\|_V^2\mbox{ for all }u\in V.
		\]
	\end{itemize} 
\end{thm}
We precede the proof by a lemma. The arguments are similar to \cite[Theorem 4.2 (i)$\Rightarrow$ (ii)]{ace}. To be complete, we give them (in a more concise way). 
\begin{lem}\label{lem:5.5}
	If $a$ is essentially real-coercive, then there exists a finite dimensional subspace $V_1\subset V$ such that $a$ is real-coercive on $V_1^\perp \times V_1^\perp$. 	
\end{lem}
\begin{proof}
	Let $P_n$ be orthogonal projections of finite rank converging strongly to the identity. We claim that there exist $\alpha>0$ and $n\in\N$ such that  
	\begin{equation}\label{neq:54}
	|\re a(u,u)| +\|P_nu\|^2_V\geq \alpha\|u\|_V^2\mbox{ for all }u\in V.
	\end{equation} 	
	In fact, if not, we find $u_n\in V$ such that $\|u_n\|_V=1$ and 
	\begin{equation}\label{eq:5.6}
	|\re a(u_n,u_n)| +\|P_nu_n\|^2_V < \frac{1}{n}.
	\end{equation}
	We may assume that $u_n\rightharpoonup u$ in $V$, taking a subsequence otherwise. Then $P_n u_n\rightharpoonup u$ as $n\to\infty$. In fact, let $v\in V$. Then $P_n v\to v$. Hence 
	\[  \langle P_n u_n,v\rangle_V =\langle u_n, P_n v\rangle_V\to \langle u,v\rangle_V.   \]
	Since $\|P_n u_n\|_V\to 0$ by (\ref{eq:5.6}), it follows that $u=0$. Thus $| \re a(u_n,u_n)|\to 0$ by  (\ref{eq:5.6}), contradicting essential real coerciveness. Thus (\ref{neq:54}) holds. If we choose $V_1:=P_n V$, then $V_1^\perp =\ker P_n$ and by  (\ref{neq:54}) 
	\begin{equation}\label{eq:5.7}
	|\re a(u,u)|\geq \alpha \|u\|_V^2\mbox{ for all }u\in V_1^\perp.
	\end{equation} 
\end{proof}	

\begin{proof}[Proof of Theorem~\ref{th:54}]
	$(i)\Rightarrow (ii)$ By Lemma~\ref{lem:5.5} there exists a finite dimensional subspace $V_1$ of $V$ such that 	(\ref{eq:5.7}) holds.
	By Proposition~\ref{prop:51} b) two cases may occur. \\
	First case:  
	\[ -\re a(u,u)\geq \alpha \|u\|_V^2\mbox{ for all }u\in V_1^\perp.   \]
	Since $\dim V_1^\perp =\infty$, there exist $u_n\in V_1^\perp$ such that $\|u_n\|_V=1$ and $u_n\rightharpoonup 0$. Since $\re a(u_n,u_n)\leq -\alpha\leq 0$, this contradicts essential positive coerciveness. Thus this case is excluded, we are in the second case and $(ii)$ is proved. \\
	$(ii)\Rightarrow (iii)$ Denote by $P$ the orthogonal projection  onto $V_1$ and by $Q:=\Id -P$ the one onto $V_1^\perp$. Thus 
	\begin{equation*}
	a(Qu,Qu)\geq \alpha \|Qu\|_V^2\mbox{ for all }u\in V.
	\end{equation*} 
	Let 
	\[  b(u,v)=a(Qu,Pv) +  a(Pu,Qv) + a(Pu,Pv) \mbox{ for all } u,v\in V.\]
	Then $a(u,v)=a(Qu,Qv)+b(u,v)$ for all $u,v\in V$. There exists a finite rank operator $K_1:V\to V'$ such that 
	\[ b(u,v)=\langle K_1 u,v\rangle \mbox{ for all } u,v\in V.   \]
	Then define $K:V\to V'$ by 
	\[   \langle Ku,v\rangle  = -\langle K_1 u,v\rangle +\alpha \langle Pu,v\rangle_V.  \]
	Then $K$ has finite rank and 
	\begin{eqnarray*}
		\re a(u,u) & =&  \re (a(Qu,Qu)+\alpha \langle Pu,u\rangle_V) \\
		& \geq &  \alpha (\| Qu\|_V^2 +\alpha \|Pu\|_V^2)\\
		& = & \alpha \|u\|_V^2.
	\end{eqnarray*}
	$(iii)\Rightarrow (iv)$ is trivial.\\
	$(iv)\Rightarrow (i)$ Let $u_n\rightharpoonup 0$ in $V$ such that $\limsup_{n\to\infty}\re a(u_n,u_n)\leq  0$. Since $K$ is compact, $Ku_n\to 0$ in $V'$ and thus $\langle Ku_n,u_n\rangle\to 0$. Thus 
	\[  \limsup_{n\to\infty} \alpha \|u_n\|_V^2\leq   \limsup_{n\to\infty} \re a(u_n,u_n)\leq 0.  \]
	$(v)\Rightarrow (i)$ The proof is the same since $u_n\rightharpoonup 0$ in $V$ implies $\|Ku_n\|_Y\to 0$ as $n\to\infty$. \\
	$(iii)\Rightarrow (v)$ Let $j:V\to V'$ be the Riesz isomorphism. Then \[J=j^{-1}\circ K:V\to V\] is of finite rank and 
	\[     \re  a(u,u) +\re \langle Ju,u\rangle_V\geq \alpha \|u\|_V^2.   \]
	Let $K_1=\frac{1}{2} (J+J^*)$. Then $\re \langle Ju,u\rangle_V=\re \langle K_1 u,u\rangle_V$ and $K_1$ is self-adjoint and of finite rank since $J$ is of finite rank. Thus there exist orthonormal vectors $e_1,\cdots,e_n\in V$ and $\lambda_k\in \R$ such that  
	\[  K_1 u=\sum_{k=1}^n \lambda_k \langle u,e_k\rangle_V e_k  . \]
	Choose $\lambda =\max \{  \lambda_1,\cdots,\lambda_n \}$ and let $K_2=\sqrt{\lambda}P$ where 
	\[   Pu:=\sum_{k=1}^n \langle u,e_k\rangle_V e_k  \]
	defines an orthogonal projection of finite rank. Then 
	\begin{eqnarray*}
		\|K_2 u\|_V^2  & =  &\lambda \|Pu\|_V^2=\lambda \langle Pu,u\rangle_V= 	\lambda \sum_{k=1}^n |\langle u,e_k\rangle_V|^2 \\
		& \geq  &  \sum_{k=1}^n \lambda_k |\langle u,e_k\rangle_V|^2
		=  \langle K_1 u,u\rangle_V. 
	\end{eqnarray*}	 
	Thus 
	\[  \re a(u,u) +\|K_2 u\|_V^2\geq \alpha \|u\|_V^2. \]
\end{proof}
Next we want to characterize essential real coerciveness. It is obvious that $-a$ is real-coercive if and only if $a$ is real-coercive. Moreover, essential positive coerciveness is stronger than real coerciveness. Recall once more that we assume that $\dim V=\infty$. 
\begin{thm}\label{th:erc}
	Assume that $a$ is essentially real-coercive. 
	Then either $a$ is essentially positive-coercive or $-a$  is essentially positive-coercive. 
\end{thm}
\begin{proof}
	a) First we show that the two assertions  are exclusive. Assume that $a$ and $-a$ are essentially positive-coercive. By Theorem~\ref{th:54} there exist $\alpha_j>0$  and compact operators $K_j:V\to V'$  ($j=1,2$) such that
	\[ \re ((-1)^j a(u,u) +\langle K_j u,u\rangle)\geq \alpha_j \|u\|_V^2.   \]	
	Adding these two inequalities, we deduce that 
	\begin{equation*}\label{eq:5.8}
	\re \langle (K_1+K_2)u,u\rangle \geq (\alpha_1+\alpha_2)\|u\|_V^2 \mbox{ for all }u\in V.
	\end{equation*}
	It follows that $K_1+K_2:V\to V'$ is invertible and compact, which is impossible since $\dim V=\infty$. \\
	b) By Lemma~\ref{lem:5.5} there exist $\alpha>0$ and  a finite dimensional subspace $V_1$ of $V$  such that 
	\[  |\re a(u,u)|\geq \alpha \|u\|_V^2  \mbox{ for all }u\in V_1^\perp .   \]
	By Proposition~\ref{prop:51}, two cases can occur : 
	\[ \re a(u,u)\geq \alpha \|u\|_V^2\mbox{ for all }u\in V_1^\perp    \] 
	or 
	\[ -\re a(u,u)\geq \alpha \|u\|_V^2\mbox{ for all }u\in V_1^\perp    \] 
	In the first case $a$ is essentially positive-coercive, by Theorem~\ref{th:54}, and in the second case $-a$ is essentially positive-coercive. 
\end{proof}	
\section{Asymptotic compactness and compact perturbation of forms}\label{sec:4}
In this section we study when a semigroup approaches a finite dimensional semigroup as $t\to\infty$. We call this property \emph{asymptotic compactness}. Our main result is concerned with compact perturbation of forms for which we show that they preserve asymptotic compactness. This section is of independent interest.  

Now assume that $-A$ generates a $C_0$-semigroup $S$ on a complex Banach space $X$. Assume that $\sigma_1$ is a compact and  relatively open subset of $\sigma(A)$. Then there 
exists a unique decomposition 
\begin{equation}\label{eq:new41}
X=X_1\oplus X_2,
\end{equation}
where $X_j$ are closed subspaces such that $S(t)X_j\subset X_j$, such that $A_1$ is bounded and $\sigma (A_1)=\sigma_1$, $\sigma(A_2)=\sigma(A)\backslash \sigma_1$, where $-A_j$ is the generator of $S_{|X_j}$ for $j=1,2$. We refer to \cite[A-III Theorem~3.3]{Na86}. The projection $P_{\sigma_1}$ onto $X_1$ along (\ref{eq:new41}) is called the \emph{spectral projection}  associated with $\sigma_1$. If $\lambda$ is an isolated point, we call $P_\lambda:=P_{\{\lambda\}}$ \emph{the spectral projection} associated with 	$\lambda$. 

Let $A$ be a closed operator on $X$. We say that $A$ is a \emph{Fredholm operator} if $\ker A$ and $X/\mathop{range}(A)$ have finite dimension. This implies that $\mathop{range}(A)$ is 
closed in $X$. By 
\[  \rho_F(A):=\{  \lambda\in\C:\lambda\Id-A\mbox{ is a Fredholm operator}  \}  \]
we denote the \emph{Fredholm resolvent set} of $A$. It is an open subset of $\C$ and we denote by $\sigma_{ess}(A):=\C\backslash \rho_F(A)$ the \emph{essential spectrum} of $A$. 

The following property is remarkable (see \cite[p. 243]{Ka80} or \cite[1.3.1]{MM}).
\begin{prop}\label{prop:4.1}
Let $A$ be the negative generator of a $C_0$-semigroup and let $\omega\subset \rho_F(A)$ be an open connected set. If $\omega\cap\rho(A)\neq \emptyset$, then $\omega\cap \sigma(A)$ is discrete and 
$P_\lambda$ has finite rank for all $\lambda\in\omega\cap\sigma(A)$.  
\end{prop}  
There are several different definitions of the essential spectrum (see \cite{EE} for 5 definitions). For example, in  \cite{RS}, the essential spectrum is the complement in $\sigma(A)$ of the set of all isolated points in $\sigma(A)$ with spectral projection of finite rank. For selfadjoint operators this coincides with our definition here, by Proposition~\ref{prop:4.1}, which also shows that the notion of spectral radius is independent of the definition.        
 
For $X$ a Banach space and  $T\in\LX$,  we let \[  \|T\|_{ess} :=\inf_{K\in\KX} \|T-K\|, \]
where ${\mathcal K}(X)$ is the closed ideal of $\LX$ consisting   of all compact operators.  
The \emph{Calkin algebra} $\LX /\KX$ is a Banach algebra for the norm
\[   \| \tilde{T}\|:=\|T\|_{ess}\]
where $T\mapsto \tilde{T}:\LX\to \LX /\KX$
 is the quotient mapping.   As is well-known, one has 
 \[   \sigma_{ess}(T)=\sigma(  \tilde{T}), \]
 where $\sigma(\tilde{T})$ denotes the spectrum of $\tilde{T}$ in the Calkin algebra. We denote by 
 \[r_{ess}(T)=  \sup\{ |\lambda|:\lambda\in \sigma_{ess}(T) \}\]
 the \emph{essential spectral radius} of $T$. 
 
 Now let $S$ be a $C_0$-semigroup on $X$ and $-A$ its generator. We denote by 
 \[\omega_{ess}(A):=\inf\{     w\in\R:\exists M\geq 0 \mbox{ such that }\|S(t)\|\leq Me^{wt}\mbox{ for all }t\geq 0\}\]       
 the \emph{essential growth bound} of $S$. Thus
  \[    \omega_{ess}(A)=\limsup_{t\to\infty}\frac{1}{t}\log \|S(t)\|_{ess} =\inf_{t>0} \log \|S(t)\|_{ess},   \]
which implies that 
\begin{equation*}\label{eq:4.3}
r_{ess}(S(t))=e^{t\omega_{ess}(A)}\mbox{ for all }t>0,
\end{equation*}  
 see  \cite[A III.]{Na86}. 

We  call the semigroup  $S$ \emph{asymptotically compact} if $\omega_{ess}(S)<0$.   Here we deviate from the terminology in \cite{EN00} and \cite{Na86} where \emph{quasi-compact} is used instead.  
 Recall that $S$ is called \emph{uniformly exponentially stable} if there exist $\varepsilon >0$, $M\geq 1$ such that 
\[    \|S(t)\|  \leq Me^{-\varepsilon t}\quad (t\geq 0).  \]
Asymptotic compactness can be characterized as follows. 
\begin{prop}\label{prop:n4.2}
Let $S$ be  $C_0$-semigroup on $X$ with generator $A$. The following assertions are equivalent:
\begin{itemize}
	\item[a)] $S$ is asymptotically compact; 
    \item[b)] there exists a decomposition $X=X_1\oplus X_2$ where $X_j$, $j=1,2$ are closed subspaces which are invariant under $S$ such that $\dim X_1<\infty$ and $S_2$ is uniformly exponentially  stable, where $S_2(t)={S(t)}_{|X_2}$.    
\end{itemize}     
\end{prop}  
\begin{proof}
 $b)\Rightarrow a)$ Since $\|S(t)\|_{ess}\leq \|S_2(t)\|\leq Me^{-\varepsilon t}$, it follows that $\omega_{ess}(A)<0$. \\
 $a)\Rightarrow b)$ Since $\omega_{ess}(A)<0$, one has $r_{ess}(S(1))=e^{\omega_{ess}(S)}<1$. Let $r\in (r_{ess}(S(1)),1)$ and note that  $\sigma_1:=\{  \lambda\in \sigma(S(1)):|\lambda|>r\}$  is finite and the spectral projection $P$ for $S(1)$ associated with $\sigma_1$ has finite rank. Then $PX=:X_1$ and $X_2:=\ker P$ define a decomposition with the desired properties (cf. \cite[A III Corollary 3.5]{Na86}).  
\end{proof}	
Thus, a semigroup $S$ is asymptotically compact if and only if it converges to a finite dimensional semigroup as $t\to\infty$, and this exponentially fast. This implies that the qualitative behaviour of $S(t)$ when $t\to\infty$ is determined by a finite dimensional system. 

We would like to add the following property which is basically a corollary of Proposition~\ref{prop:n4.2} (cf. \cite[B-IV Theorem 2.10]{Na86}).   
\begin{prop}\label{prop:n4.3}
Let $S$ be an asymptotically compact $C_0$-semigroup with generator $-A$. Then the set $\sigma_{-}(A):=\{  \lambda\in\sigma(A):\re \lambda\leq 0\}$ is finite and the spectral projection associated with $\sigma_{-}(A)$ has finite rank.
\end{prop}
Next we state a perturbation theorem due to Desch-Schappacher which will be needed later. 
\begin{thm}{\cite[Theorem~3.7.25]{ABHN11}}\label{th:3.2}
Let $A$ be the generator of a holomorphic $C_0$-semigroup $S$ and let $K:D(A)\to X$ be compact (where $D(A)$ carries the graph norm). Then $A+K$ generates a holomorphic $C_0$-semigroup $\tilde{S}$ and $\omega_{ess}(A+K)=\omega_{ess}(A)$. 
\end{thm}    
We will also need the following interpolation result which is of independent interest.
\begin{thm}\label{prop:3.3}
Let $A$ be the generator of a holomorphic $C_0$-semigroup $T$ on the Banach space $X$. Let $Y$ be a Banach space such that 
\[  D(A)\subset Y\hookrightarrow X . \]
Suppose that the part $B$ of $A$ in $Y$ generates a holomorphic $C_0$-semigroup $S$ on $Y$. Then $\omega_{ess} (A)=\omega_{ess}(B)$.  
\end{thm}
Here $B$ is defined by 
\[  D(B):=\{  y\in D(A):Ay\in Y\} ,\quad By:=Ay. \] 
Note that if $\lambda\in \rho(A)$, then $\lambda\in \rho (B)$ and $R(\lambda, B)=R(\lambda, A)_{|Y}$. 

For the proof of Theorem~\ref{prop:3.3} we need some preparation. 
\begin{lem}\label{lem:3.4}
Let $T\in \LX$ and $Y\hookrightarrow X$. Assume that $TX\subset Y$ and let $S:=T_{|Y}\in \LY$. Then 
\[   \sigma(S)\subset   \sigma(T)\cup \{0\}. \]
\end{lem}
\begin{proof}
Let $\lambda\in\rho(T)$, $\lambda\neq 0$, $y\in Y$, $x=R(\lambda, T)y$. Then $\lambda x-Tx=y$. Hence $\lambda x=y+Tx\in Y$. So $x\in Y$. Thus $\lambda\in \rho(S)$ and $R(\lambda, S)=R(\lambda, T)_{|Y}$. 	
\end{proof}	

\begin{lem}\label{lem:3.5}
	Let $T\in \LX$ and $Y\hookrightarrow X$. Assume that $TX\subset Y$ and let $S:=T_{|Y}\in \LY$. Then 
	\[   r_{ess}(S)\leq r_{ess}(T). \]
\end{lem}
\begin{proof}
Let $r>r_{ess}(T)$. Then, by Proposition~\ref{prop:4.1}, the set $M:=\{  \lambda \in \sigma(T):|\lambda|>r \}$ is finite and consists of isolated eigenvalues with finite dimensional spectral projection. Thus $X=X_1\oplus X_2$ where $X_j$ is a closed subspace, $TX_j\subset X_j$  ($j=1,2$), $\dim X_1<\infty$ and $r(T_2)<r$ where $T_2=T_{|X_2}$. Let $S=T_{|Y}$ and note that $Y_2:=X_2\cap Y$ is a closed subspaces of $Y$ invariant by $S$. Let $S_2=S_{|Y_2}$. Then $T_2X_2\subset Y_2$. It follows from Lemma~\ref{lem:3.4} that 
$\sigma(S_2)\subset \sigma(T_2)\cup \{0\}$. Hence $r(S_2)\leq r(T_2)<r$. 

Now let $|\lambda|>r$. We show that $\lambda\Id -S$ is a Fredholm operator. It is clear that $\ker (\lambda\Id -S)\subset \ker (\lambda\Id -T)$ has finite dimension. Since $S_2$ is invertible, 
\[(\lambda\Id-T)Y\supset (\lambda\Id -S_2)Y_2=Y_2.\]
We show that $Y_2$ has finite codimension in $Y$ (which in turn, implies that $(\lambda\Id -T)Y$ has finite codimension in $Y$). Since $X_2$ has finite codimension there exist $\varphi_1,\cdots,\varphi_n\in X'$ such that 
\[    X_2=\bigcap_{j=1}^n\ker \varphi_j. \]
This implies that $Y_2=  \bigcap_{j=1}^n\ker {\varphi_j}_{|Y}$, which proves the claim.  
\end{proof}	
\begin{rem}\label{rem:4.8}
The proof of Lemma~\ref{lem:3.5} yields a stronger assertion, namely 
\[\widetilde{\sigma}_{ess}(S)\subset   \widetilde{\sigma}_{ess}(T)\cup\{0\},\]
where $\widetilde{\sigma}_{ess}(T)=\sigma(T)\backslash \{ \lambda\in \sigma(T):\lambda$ is an isolated point with spectral projection of finite rank$\}$.
\end{rem}
\begin{proof}[Proof of Theorem~\ref{prop:3.3}]
a) Since $T$ is holomorphic, $T(1)X\subset D(A)\subset Y$. One has $S(t)=T(t)_{|Y}$ (since $S(t)y=\lim_{n\to\infty}\left( \Id +\frac{t}{n}B\right)^{-n}y=T(t)y$). It follows from Lemma~\ref{lem:3.5} that 
\[ e^{\omega_{ess}(B)}=r_{ess}(S(1))\leq r_{ess}(T(1))=e^{\omega_{ess}(A)}.\]
 Hence $\omega_{ess}(B)\leq \omega_{ess}(A)$.  	\\
b) Consider $Z:=D(A)$ with the graph norm. Then $T_1(t):=T(t)_{|Z}$ is a $C_0$-semigroup which is similar to $T$. Its generator $A_1$ is the part of $A$ in $Z$. Because of the similarity we have $\omega_{ess}(A_1)=\omega_{ess}(A)$. It follows from the closed graph theorem that $Z\hookrightarrow Y$. Since $T_1(t)=T(t)_{|Z}=S(t)_{|Z}$, it follows from a) that $\omega_{ess}(A_1)\leq \omega_{ess}(B)$.    
\end{proof}	
Next we want to consider semigroups associated with  a form. Let $V,H$  be  Hilbert spaces over $\C$ such that $V\underset{d}{\hookrightarrow}H$. Let  $a:V\times V\to \C$ be a continuous sesquilinear form. As before we define an operator $A$ on $H$ by 
\[ D(A):=\{ u\in V:\exists f\in H,a(u,v)=\langle f,v\rangle_H\mbox{ for all }v\in V\},\;\; Au:=f. \]
We call $A$ \emph{the operator associated with $a$ (on $H$)} and write $A\sim a$. 

The form $a$ is  called \emph{H-elliptic} if there exist $w\geq 0$, $\alpha>0$ such that  
\[ \re a(u,u)   +w\|u\|_H^2\geq \alpha \|u\|_V^2\mbox{ for all }u\in V. \]
Note that a continuous $H$-elliptic form is the same as a closed sectorial form in the terminology of Kato \cite{Ka80}. 

If $a$ is continuous and $H$-elliptic, then the associated operator $A$  is \emph{m-sectorial}, i.e., $-A$ generates a holomorphic $C_0$-semigroup 
$S:\Sigma_\theta\to\LX$ satisfying 
 \[ \|S(t)\|\leq e^{w|z|} \mbox{ for all }z\in \Sigma_\theta  \]
 for some $w\in\R$ and where $\theta\in (0,\pi/2]$, $\Sigma_\theta =\{ re^{i\alpha} :r>0,|\alpha|<\theta  \} $.  Moreover, each m-sectorial operator can be obtained in this way (and the space $V$ as well as the form $a:V\times V\to\C$ such that $a \sim A$ are unique). We refer to \cite{Ka80}. 
 
 Moreover, there is a natural operator ${\mathcal A}$ on $V'$ associated with $a$, namely by defining $D({\mathcal A})=V$ and 
 \[  \langle \A u,v\rangle =a(u,v).\] 
 Then also $-\A$ is the generator of a holomorphic $C_0$-semigroup $\mathcal S$ on $V'$ (which might no longer be quasi-contractive, see \cite{Are04}). Moreover, ${\mathcal S}(t)_{|H}=S(t)\quad (t\geq 0)$. It follows from Theorem~\ref{prop:3.3} that $\omega_{ess} (\A)=\omega_{ess}(A)$. This will be needed in the next perturbation result. This in turn is crucial for characterizing those operators  which are associated with an essentially coercive form on $H$. 
 
 \begin{thm}\label{th:3.6}
 Let $V,H$ be complex Hilbert spaces such that $V\underset{d}{\hookrightarrow} H$   and let $a:V\times V\to\C$ be a continuous $H$-elliptic form, $A\sim a$. Let  $K:V\to V'$ be compact  and define $b:V\times V\to\C$ by 
 \[ b(u,v)=a(u,v)+\langle Ku,v\rangle .   \]
 \begin{enumerate}
 \item 	Then the form $b$ is continuous and  $H$-elliptic.  
 \item Let $B\sim b$.
 Then 
   $\omega_{ess}(B)=\omega_{ess}(A)$. 
 \end{enumerate}
 \end{thm}
 \begin{proof}
 (1) We can assume that 
 \[  \re a(u,u)\geq \alpha \|u\|_V^2 \quad (u\in V),  \]
 where $\alpha>0$ (replacing $a$ by $a(.,.)+w\langle.,.\rangle_H$ otherwise). Assume that $b$ is not  $H$-elliptic. Then there exists $u_n\in V$ such that $\|u_n\|_V=1$ and 
 \begin{equation}\label{eq:3.4}
 \re a(u_n,u_n)+\re \langle K u_n,u_n\rangle +n\|u_n\|_H^2 <\frac{1}{n}.
 \end{equation} 	
 Passing to a subsequence, we may assume that $u_n\rightharpoonup u$ in $V$. Since $K:V\to V'$ is compact, it follows that $K u_n\to Ku$ in $V'$. Hence $\langle Ku_n,u_n\rangle \to \langle Ku,u\rangle$. Thus (\ref{eq:3.4}) implies that $\|u_n\|_H^2\to 0$ as $n\to\infty$. Since $V\underset{d}{\hookrightarrow} H$ it follows that $u=0$. Hence $\re \langle Ku_n,u_n \rangle \to 0$ as $n\to\infty$.  
 Thus (\ref{eq:3.4}) contradicts that  $\re a(u_n,u_n)\geq \alpha>0$ for all $n\in\N$. \\
 (2) Let $\mathcal T$ be the $C_0$-semigroup generated by $\A$ on $V'$, and $\mathcal S$ the semigroup generated by $\A+K$ on $V'$ (where $D(\A +K)=D(\A)=V$). It follows from  Theorem~\ref{th:3.2} that $\omega_{ess}(\A)=\omega_{ess}(\A +K)$. Moreover, by Theorem~\ref{prop:3.3}, $\omega_{ess}({\mathcal A})=\omega_{ess}(A)$ and $\omega_{ess}({\mathcal A}+K)=\omega_{ess}(B)$. Thus  $\omega_{ess}(A)=\omega_{ess}(B)$. 
 \end{proof}	
 
\section{Essentially positive coercive-forms and asymptotic compactness}\label{sec:6}
In this section we show that the semigroup associated with a  continuous elliptic form is asymptotically  compact if (and basically only if) the form is essentially positive-coercive. We first prove that essential positive coerciveness implies already ellipticity. This could be derived from \cite[Lemma 4.14]{AtEKS14} together with Theorem~\ref{th:54}.   However, the following proof is more direct. 

Throughout this section we consider  $V,H$  complex  Hilbert spaces such that $V\underset{d}{\hookrightarrow} H$ and $a:V\times V\to \C$ is sesquilinear and continuous.
\begin{prop}\label{prop:6.1}
 If $a$ is essentially positive-coercive, then $a$ is $H$-elliptic, i.e. there exist $\alpha>0$ and $w\geq 0$ such that 
 \[ \re a(u,u)+w\|u\|_H^2\geq \alpha \|u\|_V^2\mbox{ for all }u\in V.   \]
 \end{prop}	
\begin{proof}
If $a$ is not $H$-elliptic, then there  exist $u_n\in V$ such that $\|u_n\|_V=1$ and 
\begin{equation}\label{eq:6.2}
\re a(u_n,u_n)+n\|u_n\|_H^2<\frac{1}{n}.
\end{equation} 	
The continuity of $a$ implies that  $\re a(u_n,u_n)\geq -M$. Hence by (\ref{eq:6.2}), 
\[  \|u_n\|_H^2 \leq  \frac{1}{n} \left( M+\frac{1}{n} \right) .\]
It follows that $\lim_{n\to\infty} u_n=0$ in $H$. Since $V$ is reflexive this implies that $u_n\rightharpoonup 0$ in $V$. Now (\ref{eq:6.2}) yields a contradiction to essential positive-coerciveness.  
\end{proof}	
Let $a$ be $H$-elliptic form with associated operator $A$ on $H$. Denote by $S$ the semigroup generated by $-A$ on $H$. Thus $S$ is holomorphic. We want to study the asymptotic behaviour of $S(t)$ as $t\to\infty$.   
\begin{rem}\label{rem:62}
If $w=0$, i.e.  if $a$ is positive-coercive, then $a$ is uniformly exponentially stable. In fact there exists $c_H>0$ such that 
\[  \|u\|_H\leq c_H \|u\|_V\mbox{ for all }u\in V.   \]
Choose $\varepsilon =\frac{\alpha}{c_H^2}$. Then 
\[   \re  a(u,u)-\varepsilon \|u\|_H^2\geq \alpha \|u\|_V^2-\varepsilon c_H^2 \|u\|_V^2=0\mbox{ for all }u\in V.  \]
Thus $A-\varepsilon\Id$ is accretive. It follows that the semigroup generated by $-A+\varepsilon\Id$, i.e. $(e^{\varepsilon t}S(t))_{t\geq 0}$ is contractive. Thus 
$\|S(t)\|\leq e^{-\varepsilon t}$ for all $t\geq 0$.  
\end{rem}	
Thus positive-coercive forms lead to exponentially stable semigroups. We show now that essentially positive-coercive forms generate 
asymptotically compact semigroups. 
\begin{thm}\label{th:6.3}
	Let $a$ be an essentially positive-coercive form. Let $A\sim a$ and denote by $S$ the semigroup generated by $-A$. Then $S$ is asymptotically compact. 
\end{thm}  
\begin{proof}
By Theorem~\ref{th:54} there exists a compact operator $K:V\to V'$ such that the form $b$ is positive-coercive, where 
\[ b(u,v)=a(u,v) +\langle K u,v \rangle \mbox{ for all }u,v\in V.  \]
Let $B\sim b$ on $H$ and let $T$ be the semigroup generated by $-B$. Then $T$ is uniformly exponentially stable by Remark~\ref{rem:62}. Thus $\omega_{ess}(T)<0$. It follows from Theorem~\ref{th:3.6}  that $\omega_{ess}(S)=\omega_{ess}(T)$. Thus $S$ is asymptotically compact.  	
\end{proof}	
\begin{cor}\label{cor:6.4}
	Let $A\sim a$ where $a$ is an essentially positive-coercive form. Then there exists $\varepsilon >0$ such that 
	\[  \sigma_{ess}(A)\subset \{ \lambda\in\C : \re (\lambda)\geq\varepsilon  \}.\]
	 In particular, each $\lambda\in\sigma(A)$ with $\re (\lambda)<\varepsilon$ is an isolated point of the spectrum with finite dimensional spectral projection. 
\end{cor}
Next we answer the question which selfadjoint operators are associated with symmetric essentially coercive forms which remained open in Section~\ref{sec:3}.
\begin{cor}\label{cor:6.5}
Let $A$ be a closed operator on $H$. The following assertions are equivalent.
\begin{enumerate}
	\item[(i)] There exist $V\underset{d}{\hookrightarrow} H$ and a continuous, symmetric,   essentially positive-coercive sesquilinear form $a:V\times V\to \C$ such that $A\sim a$; 
	\item[(ii)] $A$ is selfadjoint and there exists $\varepsilon>0$ such that 
	\[  \sigma_{ess}(A)\subset \{  \lambda\in\C:\re (\lambda)\geq\varepsilon  \}.  \]  
\end{enumerate} 
In that case the form $a$ in (i) is unique.   
\end{cor}
\begin{proof}
$(i)\Rightarrow (ii)$ follows from 	Corollary~\ref{cor:6.4} and the selfadjointness of $A$ from Proposition~\ref{prop:6.1} or Theorem~\ref{th:new26}. \\
$(ii)\Rightarrow (i)$: it follows from $(ii)$ that there exists $\delta>0$ such that $\sigma_{\delta}(A):=\{ \lambda\in\sigma(A):\re (\lambda)\leq \delta\}$ is finite and each $\lambda\in\sigma_{\delta} (A)$ is an eigenvalue with finite dimensional eigenspace (see also \cite[VII.3]{RS}).  Denote by $P$ the spectral projection associated with $\sigma(A)\cap(-\infty, \delta)$. Then $P$ is selfadjoint of finite rank. The space $PH$ has an orthonormal basis $e_1,...,e_n$ of eigenvectors of $A$, i.e., $e_k\in D(A)$ and $Ae_k = \lambda_k e_k$ for $k=1,...,n$. Let $B$ be defined on $D(B) = D(A)$ by  $Bu = Au - \sum_{k=1}^n \lambda_k \langle u,e_k\rangle_H e_k$. Then $B$ is selfadjoint and 
$\sigma(B) \subset  [\delta,\infty)$. Thus there exist a Hilbert space $V$, $V\underset{d}{\hookrightarrow} H$ and a coercive, continuous, symmetric form $b:V\times V\to\K$ such that $B \sim b$. Define $a:V\times V\to\K$  by
\[  a(u,v) = b(u,v)  +  \sum_{k=1}^n \lambda_k \langle u,e_k\rangle_H \langle e_k,v \rangle_H. \]
Then $A\sim a$. Uniqueness follows from the fact that for each m-sectorial operator  $A$ there is a unique closed form $a$ such that $A\sim a$, see \cite[VI. Theorem~2.7]{Ka80}. 
\end{proof}	 
\begin{cor}\label{cor:6.6}
Let $A$ be a closed operator on $H$. The following assertions are equivalent. 
\begin{enumerate}
	\item[(i)] There exist $V\underset{d}{\hookrightarrow}H$ and a  symmetric, essentially coercive form $a:V\times V\to\C$ such 
	that $A\sim a$.  
	\item[(ii)] The operator $A$ is selfadjoint and there exists $\delta>0$ such that 
	\[  \sigma_{ess}(A)\subset (\delta,\infty)\mbox{ or } \sigma_{ess}(A)\subset (-\infty,-\delta).    \]
\end{enumerate}
\end{cor} 
\begin{proof}
If $a$ is symmetric, then essentially coercive is the same as essentially real-coercive and this in turn is equivalent to $a$ or $-a$ being essentially positive-coercive by Theorem~\ref{th:54}. Now Corollary~\ref{cor:6.6} follows from Corollary~\ref{cor:6.5}.  	
\end{proof}	
A useful criterion for proving essential real-coerciveness is the following.
\begin{lem}\label{lem:6.7}
Let $a:V\times V\to\C$ be of the form $a=a_0-b$ where $a_0$ is a real-coercive form and $b$ is a continuous sesquilinear form satisfying 
\begin{equation*}\label{eq:6.3}
u_n\rightharpoonup 0\mbox{ in }V\Rightarrow \limsup_{n\to\infty}\re b(u_n,u_n)\leq 0.
\end{equation*} 
Then $a$ is essentially positive-coercive. 
\end{lem} 
\begin{proof}
Let $u_n\rightharpoonup 0$ such that $\limsup_{n\to\infty}\re a(u_n,u_n)\leq 0$. Then, for some $\alpha>0$, 
\begin{eqnarray*}
\limsup_{n\to\infty} \alpha \|u_n\|_V^2  & \leq & 	\lim_{n\to\infty} \re a_0(u_n,u_n)\\
 & = & \limsup_{n\to\infty} (\re a(u_n,u_n) +\re b(u_n,u_n))\\
  & \leq & 0.
\end{eqnarray*} 
\end{proof}	
\begin{rem}
If the injection $V\hookrightarrow H$ is compact, then   $a$ is $H$-elliptic if and only if $a$ is essentially positive-coercive  (see Proposition~\ref{prop:6.1} for one direction, the other is obvious). In that case, the associated semigroup $S$ consists of compact operators and so $\omega_{ess}(S)=-\infty$.  
\end{rem}
We now give an example to show how Theorem~\ref{th:6.3} can be applied. In this example, the embedding of $V$ in $H$ is not compact.
\begin{exam}\label{ex:6.9}
Let $\Omega=\R^d\backslash \overline{\omega}$ be an exterior domain where $\omega$ is a bounded Lipschitz domain and $d\geq 3$. Let $H=L^2(\Omega)$, $V=H^1(\Omega)$, and consider the form $a$ given by 
\[ a(u,v)=\int_{\Omega} \nabla u \overline{\nabla v} +\delta \int_{\Omega} u\overline{v} +\sum_{j=1}^d\int_{\Omega }(b_j u\overline{\partial_j v } +c_j \overline{v}\partial_j u) , \]
where the coefficients $b_j,c_j$ are complex-valued in $L^d(\Omega)\cap L_{loc}^{2q/(q-2)}(\Omega)$, where $2<q<2d/(d-2)$ and where $\delta>0$. 
Then the form $a$ is continuous and essentially positive-coercive. Let $A\sim a$, $S$ the semigroup generated by $-A$. Then $S$ is asymptotically compact. Note that $A$ is selfadjoint if $\overline{b_j}=c_j$ for $j=1,\cdots ,d$.   
\begin{proof}
Since $H^1(\Omega)\hookrightarrow L^{2d/(d-2)}(\Omega)$, the form $a$ is continuous. By Lemma~\ref{lem:6.7}	it suffices to show
  the following. Let $u_n\rightharpoonup 0$ in $H^1(\Omega)$. Then $\int_{\Omega} gu_n \overline{\partial_j u_n}\to 0$ as $n\to\infty$, where $g\in L^d (\Omega)\cap L_{loc}^{2q/(q-2)}(\Omega)$.   Let $\varepsilon >0$ and let $B$ be a ball large enough such that $\left( \int_{\Omega\backslash B } |g|^d\right)^{1/d}\leq \varepsilon$. Since  $(u_n)_n$ is bounded in $H^1(\Omega)$, there exists a constant $c>0$ such that $\|u_n\|_{L^{2d/(d-2)}}\leq c$ and $\|\partial_j u_n\|_{L^2}\leq c$ for all $n\in\N$.    Thus 
\begin{eqnarray*}
\left| \int_{\Omega\backslash B } g u_n \overline{\partial_j u_n} \right| & \leq & c\left(   \int_{\Omega\backslash B}|g u_n|^2 \right)^{1/2}\\
  & \leq  &	c \left(  \int_{\Omega\backslash B} |g|^{2\frac{d}{2}} \right)^{\frac{2}{d}\frac{1}{2}}    \left( \int_{\Omega} |u_n|^{2\frac{d}{d-2}}     \right)^{\frac{d-2}{d} \frac{1}{2}} \\
   & \leq & c^2 \varepsilon\mbox{ for all }n\in\N.
\end{eqnarray*}
Note that the embedding $H^1(\Omega\cap B)\hookrightarrow L^q(\Omega\cap B)$ is compact. 

Thus $u_n\to 0$ as $n\to\infty$ in $L^q(\Omega\backslash B)$. We estimate 
\begin{eqnarray*}
\left|    \int_{\Omega\cap B} gu_n \overline{ \partial_j u_n} \right|  & \leq & c {\left(  \int_{\Omega\cap B}|g u_n|^2  \right)}^{1/2}\\ 
 & \leq & c \left( \int_{\Omega\cap B} |g|^{2\frac{q}{q-2}}\right)^{\frac{q-2}{2}\frac{1}{2}} 
 \left(  \int_{\Omega\cap B} |u_n|^{2\frac{q}{2}}\right)^{\frac{2}{q}\frac{1}{2}} \\
  & \leq & c \|g\|_{L^{2q/(q-2)}(B\cap \Omega)} \|u_n\|_{L^q(\Omega\cap B)}
 \end{eqnarray*}
which converges to $0$ as $n\to\infty$. Thus 
\[  \limsup_{n\to\infty}\left|   \re \int_{\Omega} g u_n \overline{\partial_j u_n}\right|\leq c^2\varepsilon  \]
where $\varepsilon>0$ is arbitrary. This proves the claim. The remaining assertions follow from Theorem~\ref{th:6.3}.
\end{proof}
\end{exam}
\begin{rem}\label{rem:6.10}
Another condition on the coefficients is $b_j,c_j\in L^\infty(\Omega)$ and 
\[  \lim_{R\to\infty} (\|b_j\|_{L^\infty (\Omega\backslash B_R)} +\|c_j\|_{L^\infty (\Omega\backslash B_R)} ) =0, \]
where $B_R$ is the ball in $\R^d$ of radius $R$. The proof is similar.  
\end{rem}  
Our next goal is to prove the converse of Theorem~\ref{th:6.3}. We need the following result which is a consequence of a characterization of operators with bounded $H^\infty$-calculus due to C. Le Merdy \cite{LM98}. We refer also to the monography by M. Haase \cite[Sec. 7.3.3]{Haa06}. 
\begin{prop}\label{prop:6.11}
Let $S$ be a quasicontractive holomorphic $C_0$-semi\-group on a  Hilbert space $H$ whose generator is $-A$. If $\re \lambda >0$ for all $\lambda\in\sigma(A)$, then there exist a Hilbert space $V\underset{d}{\hookrightarrow} H$, a continuous real-coercive sesquilinear form $a:V\times V\to \C$, and there exists an equivalent scalar product $[.,.]$ on $H$ such that $A$ is associated with $a$ on $(H,[.,.])$.      
\end{prop}  
Note that the assertion in the proposition says that 
\[  D(A)=\{   u\in V:\exists f\in H \mbox{ such that }a(u,v)=[f,v]\mbox{ for all  }v\in V\} , \, Au=f.  \] 
To say that $[.,.]$ is an \emph{equivalent scalar product} means that $u\mapsto \sqrt{[u,u]}$ defines an equivalent norm on $H$.

\begin{proof}[Proof of Proposition~\ref{prop:6.11}]
Since $\re \lambda>0$ for all $\lambda\in \sigma(A)$ and since the spectrum of $A$ lies in a sector, there exists $\varepsilon>0$ such that $\re (\lambda)\geq 2\varepsilon$ for all $\lambda\in\sigma(A)$. This implies that $-A+\varepsilon\Id$ generates a holomorphic $C_0$-semigroup $T$ which is bounded on a  sector.   Since $A+w\Id$ is $m$-accretive, it follows that $A+w\Id\in BIP(H)$ (see for example \cite{LM98} for the definition of $BIP(H)$). By \cite[Corollary 2.4]{ABH01}, it follows that $A-\varepsilon\Id \in BIP(H)$. Now it follows from a result of Le Merdy \cite[Theorem 1.1 and Corollary 4.7]{LM98} that there exists an equivalent scalar product $[.,.]$ on $H$ such that $-A+\varepsilon\Id$ generates a holomorphic $C_0$-semigroup $T$ which is contractive for $[.,.]$ on a sector. By a result of Kato \cite[VI § 2, Theorem 2.7]{Ka80} there are a Hilbert space $V\underset{d}{\hookrightarrow} H$, a continuous $H$-elliptic form $b:V\times V\to\C$  and $\theta\in [0,\pi/2)$ such that $A-\varepsilon\Id$ is associated with $b$ on $(H,[.,.])$. Then $A$ is associated with $a:=b+\varepsilon [.,.]$ on $(H,[.,.])$. Since $\re b(u,u)\geq 0$ for all $u\in V$,
\[ \|u\|_1^2:=\re a(u,u)=\re b(u,u)+\varepsilon \|u\|_H^2  \]
defines a norm on $V$. Since $b$ is $H$-elliptic, this norm is complete. Thus $\|\;\|_1$ is an equivalent norm on $V$. Hence there exists $\alpha>0$ such that 
\[ \re a(u,u) =\|u\|_1^2\geq \alpha \|u\|_V^2; \]
i.e. $a$ is real-coercive.            	
\end{proof}	
\begin{thm}\label{th:6.12}
	Let $S$ be an asymptotically compact quasi-contractive holomorphic $C_0$-semigroup on $H$ with generator $-A$. Then there exist an equivalent scalar product $[.,.]$ on $H$, a Hilbert space $V\underset{d}{\hookrightarrow}H$ and a continuous, essentially positive-coercive form $a:V\times V\to\C$ such that $A\sim a$ on $(H,[.,.])$.    
\end{thm}
 \begin{proof}
 Since $S$ is quasi-compact, $H=X_1\oplus X_2$ where $X_j$ are closed invariant subspaces, $\dim X_1<\infty$ and 
 \[   \|S_2(t)\|\leq Me^{-\varepsilon t} \mbox{ for all }t\geq 0\]
 for some $\varepsilon >0$, $M\geq 1$, where $S_2(t):=S(t)_{|X_2}$. It follows   from Proposition~\ref{prop:6.11} that there exist an equivalent scalar product $[.,.]_2$ on $X_2$, $V_2\underset{d}{\hookrightarrow}X_2$, $a_2:V_2\times V_2\to\C$ a continuous, real-coercive sesquilinear form	such that $a_2\sim A_2$ on $(X_2,[.,.])$, where $-A_j$ is the generator of $S(.)_{|X_j}$, $j=1,2$. 
 
 Note that $A_1$ is a bounded operator. 
 Defining 
 \[V:=X_1\oplus V_2\mbox{ and }[x_1+x_2,y_1+y_2]:= \langle x_1,y_1\rangle_H+ [x_2,y_2]_2\mbox{ on }H\]
 where $x_1+x_2,y_1+y_2\in X_1\oplus X_2$ and defining 
 \[ a(x_1+x_2,y_1,y_2) =\langle A_1x_1,y_1\rangle_H+a_2(x_2,y_2) \mbox{ for }x_1+x_2,y_1+y_2\in X_1\oplus V_2, \]
 one obtains an essentially positive-coercive form on $V$ (Theorem~\ref{th:54} (ii) is fulfilled). Then $A\sim a$ on $(H,[.,.])$
\end{proof}

\section{Essentially coercive forms and holomorphic semigroups}\label{sec:7}
The purpose of this section is to give some  
  information on the operator associated with an essentially coercive form using our results from Section~\ref{sec:6}. This will allow us to prove a converse version of Theorem~\ref{th:new29}.  Throughout the section, $V,H$ are complex Hilbert spaces, $V\underset{d}{\hookrightarrow}H$  and $a:V\times V\to\C$ is a continuous and sesquilinear form.  
As previously, we denote by $A$ the operator on $H$ associated with $a$.
\begin{thm}\label{th:7.1}
If $a$ is essentially coercive, then there exists $\theta\in\R$ such that $e^{i\theta}A$ is the negative generator of a holomorphic $C_0$-semigroup, which is quasi-contractive on a sector and which is asymptotically compact.  In particular, $D(A)$ is dense in $V$.  
\end{thm} 
\begin{proof}
There exists $\theta\in\R$ such that $e^{i\theta}a$ is essentially positive-coercive (Proposition~\ref{prop:53}). By Proposition~\ref{prop:6.1}, the form $e^{i\theta}a$ is $H$-elliptic. Thus its associated operator  $e^{i\theta} A$ is m-sectorial. As is well-known, this implies that $D(e^{i\theta}A)=D(A)$ is dense in $V$.   	
\end{proof}	
If $a$ is accretive, we know from Subsection~\ref{sec:new4} that $-A$ itself generates a contractive $C_0$-semigroup $S$. In general $S$ is not holomorphic. Here is an example. 
\begin{exam}\label{ex:7.2}
Let $H=\ell^2$, $V=\{ x\in\ell^2: \sum_{n=1}^{\infty} \lambda_n |x_n|^2<\infty \}$ where 
\[0<\lambda_1\leq \cdots\leq \lambda_n\leq \lambda_{n+1}\to\infty\]
equipped with the norm
\[ \|x\|_V^2 =\sum_{n=1}^\infty \lambda_n|x_n|^2.  \]
Let $a(u,v)=\sum_{n=1}^{\infty} i\lambda_n u_n\overline{v_n}$. Then $\re a(u,u)=0$, so $a$ is  accretive. One has $S(t)u=(e^{-i\lambda_n t}u_n)_{n\in\N}$. Thus $S$ is not holomorphic. In fact $S$ extends to a unitary group.   
\end{exam}
However, the semigroup $S$ associated with an essentially coercive,  continuous and  accretive form is always the boundary of a holomorphic $C_0$-semigroup.  To explain this result in detail, we recall the following.
 
For $\theta\in (0,\pi/2)$, as previously, let $\Sigma_{\theta} =\{ re^{i\alpha} :r>0,|\alpha|<\theta  \} $ and  
let $S:\Sigma_{\theta}\to\LX$ be a holomorphic $C_0$-semigroup with generator $-A$. 
If 
\begin{equation}\label{eq:7.1}
\sup_{z\in\Sigma_{\theta}, |z|\leq 1} \| S(t)\|<\infty,
\end{equation}
then there exists a strongly continuous extension $ \overline{S}:\overline{\Sigma_{\theta}}\to\LX$ and $S_{\pm \theta}:=(\overline{S}(te^{\pm i\theta})_{t\geq 0}$ are $C_0$-semigroups with generator $-e^{\pm i\theta}A$. We call $S_{\pm \theta}$ the boundary semigroups of $S$. Conversely, if $A$ is a closed operator and $\theta\in (0,\pi/2)$ such that $- e^{\pm i\theta}A$ generate $C_0$-semigroups $S_{\pm}$, then $-A$ generates a holomorphic $C_0$-semigroup $S:\Sigma_{\theta}\to \LX$ where $0<\theta<\pi/2$ such that (\ref{eq:7.1}) holds and $S_{\pm}$ are the boundary semigroups of   $S$. We refer to \cite[Sec. 39]{ABHN11} or \cite{AEM97}  for this and for further information. Now we can formulate a result which is a converse version of 
Theorem~\ref{th:new29}.
\begin{thm}\label{th:7.3}
	Assume that $V\neq H$, $V\underset{d}{\hookrightarrow} H$,  $a:V\times V\to \C$ is continuous, essentially coercive and  accretive. Let $A\sim a$ and denote by $S$ the  $C_0$-semigroup generated by $A$. Then $S$ is holomorphic or $S$ is the boundary semigroup of a holomorphic $C_0$-semigroup.     
\end{thm}
\begin{proof}
By Theorem~\ref{th:7.1} there exists $\theta\in [-\pi,\pi)$ such that $-e^{i\theta}A$ generates a holomorphic $C_0$-semigroup. Then $\theta\in (-\pi,\pi)$ because otherwise  $-A$ generates a $C_0$-semigroup and hence $A$ is bounded. This contradicts the assumption $V\neq H$. 

Let $B=-e^{i\theta/2} A$. Then $-e^{\pm i\theta/2}B$ generates a  $C_0$-semigroup. Thus $-B$ generates a holomorphic $C_0$-semigroup $T$ and the semigroup generated by $-e^{-i \theta/2}B=-A$ is a boundary semigroup of $T$.  
\end{proof}	
Even though we cannot apply our results on the essential spectrum in the situation of Theorem~\ref{th:7.3} (since $a$ might not be essentially positive-coercive), we note the following result. 
\begin{thm}\label{th:7.4}
Let $V\underset{d}{\hookrightarrow}H$ and let $a:V\times V\to\C$ be continuous, accretive and essentially coercive. If $a$ is not coercive, then $0$ is an eigenvalue of $A$, where $A\sim a$.  
\end{thm} 
\begin{proof}
Since $a$ is not coercive, there exist $u_n\in V$ such that $\|u_n\|_V=1$ and $a(u_n,u_n)\to 0$ as $n\to\infty$. We may assume that $u_n\rightharpoonup u$ in $V$. Since $a$ is essentially coercive, it follows that $u\neq 0$. Let $v\in V$. Then 
\begin{eqnarray*}
\re a(u,v) & = & \lim_{n\to\infty} \re a(u_n,v)\\
  & \leq & \limsup_{n\to\infty} (\re (a(u_n,u_n))^{1/2}(\re a(v,v))^{1/2}\\
   & = & 0.	
\end{eqnarray*}	  
(Here we used the accretivity of $a$). This implies that $a(u,v)=0$ for all $v\in V$. Hence $u\in D(A)	$ and $Au=0$. 
\end{proof}	
\begin{rem}\label{rem:7.5}
a) One cannot omit the hypothesis that $a$ is accretive in Theorem~\ref{th:7.4}, see the discussion at the end of Section~\ref{sec:8}.\\
b) In the situation of Theorem~\ref{th:7.4}, the semigroup $S$ generated by $-A$ may not be asymptotically compact. In fact, as is easy to see, a bounded $C_0$-group is never asymptotically compact unless the underlying space is finite dimensional. Thus the semigroup in Example~\ref{ex:7.2} gives what we want.  	
\end{rem}
The semigroup  $S$ obtained in Theorem~\ref{th:7.1} is contractive on a sector $\Sigma_{\theta}$ and asymptotically compact, i.e. $\lim_{t\to\infty}\|S(t)\|_{ess}=0$. This implies automatically that  $\lim_{t\to\infty}\|S(te^{i\beta})\|_{ess}=0$ for all $\beta\in(-\theta,\theta)$ as we will show now. However, as Example~\ref{ex:7.2}  shows, this is not true for $\beta\in[-\theta,\theta]$. Note that the boundary semigroup is obtained by the continuous extension of $S$ to $\overline{\Sigma_{\theta}}$ for the strong operator topology and not the uniform topology. 
\begin{thm}\label{7.5}
	Let $X$ be a complex Banach space $\theta\in(0,\pi/2]$ and let $S:\Sigma_{\theta}\to\LX$ be a holomorphic $C_0$-semigroup such that 
	\[  \|S(z)\|\leq M\mbox{ for all }z\in \Sigma_{\theta}.  \]
	If $\lim_{t\to\infty}\|S(t)\|_{ess}=0$, then for all $0<\beta<\theta$, 
	\[   \lim_{|z|\to\infty,z\in\Sigma_{\beta}}\|S(z)\|_{ess}=0.\] 
\end{thm}
\begin{proof}
a) We show that $\lim_{t\to\infty} \|S(tz)\|_{ess}	=0$ for all $z\in\Sigma_{\theta}$. For that we define $F:\Sigma_{\theta}\to {\mathcal C}^b([0,\infty),\LX)$ by $F(z)(t)=S(tz)$. It follows from \cite[Theorem A.7]{ABHN11} that $F$ is holomorphic. Consider the map 
\[  q:  {\mathcal C}^b([0,\infty),\LX) \to {\mathcal C}^b([0,\infty]),\LX/{\mathcal K}(X)) \]
given by $q(f)(t)=\widehat{f(t)}$ where for $T\in\LX$, $\widehat{T}$ is the image in $\LX/{\mathcal K}(X)$ by the quotient map. Then $q$ is linear and bounded. Thus $q\circ F$ is holomorphic. Consider the Banach space ${\mathcal C}_0([0,\infty),\LX/{\mathcal K}(X))$ of all continuous functions $g:[0,\infty)\to \LX/{\mathcal K}(X)$ satisfying 
\[    \lim_{t\to\infty}   \|g(t)\|_{\LX/{\mathcal K}(X)}=0      \]
and denote by $Q:{\mathcal C}^b\to  {\mathcal C}^b/{\mathcal C}_0$ the quotient map. Then $G:=Q\circ q\circ F$ is holomorphic. Since $G(t)=0$ for all $t>0$, it follows from the Uniqueness Theorem that $G(z)=0$ for all $z\in\Sigma_{\theta}$; i.e. 
\[ \lim_{t\to\infty}\|S(tz) \|_{ess}=\lim_{t\to\infty}\|\widehat{S(tz)}\|_{\LX/{\mathcal K}(X)}=0.  \]
b) Let $0<\beta<\theta$ and $\varepsilon>0$. There exists $t_0>0$ such that $\|S(e^{\pm i\beta}t)\|_{ess}\leq \varepsilon$ for all $t\geq t_0$. Let $z=te^{i\eta}$, $t\geq t_0$, $\eta\in [0,\beta]$. Then 
\[   \|S(te^{i\eta})\|=\|S(e^{i(\eta-\beta)})S(te^{i\beta})\|\leq M\| S(te^{i\beta})\|\leq M\varepsilon. \]   
If $\eta\in[-\beta,0]$ the argument is similar. 
\end{proof}
\section{Selfadjoint operators revisited}\label{sec:8}
In this section we reconsider Theorem~\ref{th:new26} and give a complete characterization if $j$ is injective. 

Let $V,H$ be Hilbert spaces such that $V\underset{d}{\hookrightarrow}H$ and let $a:V\times V\to\C$ be a continuous sesquilinear form. We define the numerical range of $a$ (with respect to $H$) by 
\[  W(a):=\{  a(u,u):u\in V,\|u\|_H=1 \}.  \]
For $\lambda\in\C$ we define $a_\lambda:V\times V \to\C$ by 
\[  a_\lambda (u,v)=a(u,v)-\lambda \langle u,v\rangle_H .\]
This is consistent with the notations of Section~\ref{sec:3}.1 for $j=\Id$. We first give the following characterization of coerciveness in terms of the numeraical range (in $H$). 
\begin{prop}\label{prop:8.1}
Let $\lambda\in\C$. The following assertions are equivalent:
\begin{itemize}
	\item[(i)]  $a_\lambda$ is coercive;
	\item[ii)]  $\lambda\not\in \overline{W(a)}$ and $a_\lambda$ is essentially coercive.
\end{itemize}	
\end{prop}
\begin{proof}
To say that $\lambda\not\in  \overline{W(a)}$ is equivalent to the existence of $\beta>0$ such that 
\begin{equation}\label{eq:8.1}
|a_\lambda (u,u)|\geq \beta \|u\|_H^2\mbox{ for all }u\in V. 
\end{equation}	
If $a_\lambda$ is coercive, i.e. if there exists $\alpha>0$ such that 
\[ |a_\lambda(u,u)|\geq \alpha \|u\|_V^2  \mbox{ for all }u\in V,      \] 
then 
\[ |a_\lambda(u,u)|\geq \frac{\alpha}{c_H^2} \|u\|_H^2  \mbox{ for all }u\in V,      \] 
where $\|u\|_H\leq c_H\|u\|_V$ for all $u\in V$. \\
$(i)\Rightarrow (ii)$ The coercivity of $a_\lambda$  implies (\ref{eq:8.1}) with $\beta=\frac{\alpha}{c_H^2}$, which implies that $\lambda\not\in \overline{W(a)}$. Moreover, the coercivity always implies the essential coercivity. \\
$(ii)\Rightarrow (i)$ Assume that $a_\lambda$ is not coercive. Then there exist $u_n\in V$ such that $\|u_n\|_V=1$ and $a(u_n,u_n)\to 0$. Passing to a subsequence we may assume that $u_n\rightharpoonup u$ in $V$. It follows from (\ref{eq:8.1}) that $\|u_n\|_H\to 0$. Consequently $u=0$. Thus $a_\lambda$ is not essentially coercive.    
\end{proof}	   
If $a$ is symmetric, then $\overline{W(a)}\subset\R$. So Proposition~\ref{eq:8.1} applies to each $\lambda\in \C\backslash\R$ and therefore, for such $\lambda$, it is equivalent to say that $a_\lambda$ is essentially coercive or coercive.

We denote by $A$ the operator associated with $a$ on $H$. From Theorem~\ref{th:new26} we know that $A$ is selfadjoint whenever $a_\lambda$ is essentially coercive for some $\lambda\in\C$. Here is a characterization of the operators obtained in this way. 
\begin{thm}\label{th:8.2}
Let $A$ be an operator on a Hilbert space $H$. The following assertions are equivalent:
\begin{itemize}
	\item[(i)]   $A$ is selfadjoint and semibounded;
	\item[ii)]   there exist a Hilbert space $V\underset{d}{\hookrightarrow} H$ and a symmetric, continuous form $a:V\times V\to\C$ such that $a_\lambda$ is essentially coercive for some $\lambda\in\C$ and $A\sim a$.    
\end{itemize}	
In that case the form in $(ii)$ is unique.
\end{thm}     
\begin{proof}
$(i)\Rightarrow (ii)$ By the Spectral Theorem, we may assume that $H=L^2(\Omega, \mu)$ with $(\Omega, \Sigma,\mu)$ a measure space, and that $A$ is a multiplication operator, i.e. there exists a measurable function 
$m:\Omega\to\R$ such that 
\[ D(A) =\{ u\in L^2(\Omega,\mu):\; mu\in L^2(\Omega,\mu)  \},\;\; Au=mu. \] 
Define $V:=\{  u\in L^2(\Omega, \mu):\int_{\Omega} |m||u|^2d\mu<\infty  \}$. Then $V$ is a Hilbert space for the norm
\[   \|u\|_V^2:=\int_{\Omega}(1+|m|)|u|^2 d\mu   \]
and $V\hookrightarrow H$. Since $A$ is densely defined as selfadjoint operator and $D(A)\subset V$, it follows that $V$ is dense in $H$. 

Define $a:V\times V\to\C$ by 
\[  a(u,v)\int_{\Omega} m u\overline{v}d\mu.   \]
Then $a$ is a continuous, symmetric from and it is easy to see that $A\sim a$.  Now assume that $A$ is bounded below. Then $m\geq -w$ $\mu$-a.e. for some $w\geq 0$. We show that $a_i$ is coercive. Let $u_n\in V$ such that 
\[a_i (u_n,u_n)=\int_{\Omega}m|u_n|^2d\mu -i\int_{\Omega}|u_n|^2d\mu\to 0\mbox{ as }n\to\infty.\] 
Then \[\|u_n\|_H^2=-\Im a_i(u_n,u_n)\to 0.\]
Moreover, \[   \int_\Omega m|u_n|^2d\mu =\re a_i (u_n,u_n)\to 0.\]
Consequently, 
\begin{eqnarray*}
\|u_n\|_V^2 & = & \int_{\Omega} (m^++ m^-)|u_n|^2 d\mu +\int_{\Omega}|u_n|^2d\mu \\ 
& \leq & \int_{\Omega}m^+|u_n|^2d\mu +(w+1)\int_{\Omega}|u_n|^2 d\mu\\
& \leq & \int_{\Omega} 	m |u_n|^2 d\mu +(2w+1)\int_{\Omega}|u_n|^2 d\mu\\
 & & \rightarrow 0  \mbox{ as } n\to\infty.
\end{eqnarray*}	 
We have shown that $a_i(u_n,u_n)\to 0$ implies that $\|u_n\|_V\to 0$ as $n\to\infty$. This is equivalent to $a_i$ being coercive.  \\
 If $A$ is bounded above, applying the argument to $-A$, we deduce that $(-a)_i=-a_{-i}$ is coercive. Then also $a_{-i}$ is coercive and by Lemma~\ref{lem:perturb}	and Proposition~\ref{prop:8.1} also $a_i$ is coercive.\\
$(ii)\Rightarrow (i)$ We know from Theorem~\ref{th:new26} that $A$ is selfadjoint. It remains to show that 
$A$ is semibounded.  Since $a_\lambda$ is essentially coercive, it follows from Lemma~\ref{lem:perturb} that $a_i$ is essentially coercive. By Proposition~\ref{prop:53}, there exists $\theta\in\R$ such that $e^{i\theta}a_i$ is essentially  positive-coercive and so $H$-elliptic by Proposition~\ref{prop:6.1}. Thus there exist $w\in\R$, $\alpha>0$ so that 
\[   \re e^{i\theta}(a(u,u)-i\|u\|_H^2) +w\|u\|_H^2\geq \alpha \|u\|_V^2\mbox{ for all }u\in V.  \]   
Consequently 
\[   \re e^{i\theta} a(u,u)+(w+1)\|u\|_H^2\geq \alpha \|u\|_V^2\mbox{ for all }u\in V.  \]  
First case: $\re e^{i\theta}=0$. Since $V\underset{d}{\hookrightarrow} H$, it follows that the norms of $V$ and $H$ are equivalent. Thus $V=H$, i.e. $A$ is a bounded operator. \\
Second case:   $\re e^{i\theta}>0$. Then 
\[  \langle Au,u\rangle_H \geq \frac{-(w+1)}{\re e^{i\theta}}\|u\|_H^2\mbox{ for all }u\in D(A).   \]
Thus $A$ is bounded below. \\
Third case:  $\re e^{i\theta}<0$. Then the second case implies that $-A$ is bounded below, i.e. $A$ is bounded above, which proves $(i)$. 

We now prove uniqueness. Let $A$ be a selfadjoint operator which is bounded below, say
\[  \langle Au,u\rangle_H\geq (-w+1)\|u\|_H^2\mbox{ for all }u\in D(A) \mbox{ and some }w\geq 0.  \]
Let $A\sim a $ where $a:V\times V\to\C$ is a symmetric, continuous form, which is essentially  coercive. Since, by Theorem~\ref{th:7.1}, $D(A)$ is dense in $V$, it follows that 
\[ a(u,u)+w\|u\|_H^2\geq \|u\|_H^2\mbox{ for all }u\in V.  \]
Thus $0\not\in \overline{W(a_{-w})}$. It follows from Lemma~\ref{lem:perturb} that $a_{-w}$ is essentially coercive. Now Proposition~\ref{prop:8.1}  implies that $a_{-w}$ is coercive. Thus $a$ is $H$-elliptic, or equivalently, $a$ is a sectorial closed form.  By \cite[Theorem~2.7]{Ka80}, there is only one sectorial closed form to which $A$ is associated.     
 \end{proof}	
We illustrate Theorem~\ref{th:8.2}  by considering a lower bounded selfadjoint operator $A$ on a Hilbert space $H$. Let \[\lambda_1:=\inf\{   \langle Au,u\rangle_H:u\in D(A),\|u\|_H=1\}\]
 be the \emph{lower bound} of $A$. Assume that there exists $\lambda_{ess}>\lambda_1$ such that $\sigma_{ess}(A)=[\lambda_{ess},\infty)$. Denote by $a$ the unique form of Theorem~\ref{th:8.2} such that $A\sim a$. Then the following assertions hold.\\
 a) The form $a_\lambda$ is coercive if and only if $\lambda\in\C\backslash\R$ or $\lambda\in\R$, $\lambda<\lambda_1$. \\
 b) For $\lambda\in\R$, $a_\lambda$ is essentially coercive if and only if $\lambda < \lambda_{ess}$. \\
 c) The lower bound $\lambda_1$ belongs to $\sigma_p(A)$. \\
 d)  If $\lambda\in (\lambda_1,\lambda_{ess}) \cap \rho(A)$, then $a_\lambda$ is essentially coercive but not coercive. Observe that $A-\lambda\Id \sim a_\lambda$, but $0\not\in \sigma_p(A-\lambda\Id)$ (cf. Theorem~\ref{th:7.4}). \\
 \begin{proof}
 a) follows from Proposition~\ref{prop:8.1}.\\
 b) follows from Corollary~\ref{cor:6.4}.\\
 c) follows from Theorem~\ref{th:7.1} and is well-known. \\
 d) is a consequence of b). 	
\end{proof}

\noindent \textbf{Acknowledgments:}  This research is partly supported by the B\'ezout Labex, funded by ANR, reference ANR-10-LABX-58.   

\bibliographystyle{abbrv}

\end{document}